\documentclass[a4paper,11pt]{amsart}

\usepackage{fancybox}
\usepackage{amscd}
\usepackage{amsmath}
\usepackage{amssymb}
\usepackage{amsthm}
\usepackage{float}
\usepackage[dvips, pdftex]{graphicx}
\usepackage{tikz}
\usepackage[all, ps, dvips, pdftex]{xy}
\usepackage{color}
\usepackage[linktocpage,hidelinks, hyperfootnotes=false]{hyperref}
\usepackage{array}
\usepackage{amscd}

\DeclareFontFamily{U}{rsfs}{%
\skewchar\font127}
\DeclareFontShape{U}{rsfs}{m}{n}{%
<-6>rsfs5<6-8.5>rsfs7<8.5->rsfs10}{}
\DeclareSymbolFont{rsfs}{U}{rsfs}{m}{n}
\DeclareSymbolFontAlphabet
{\mathrsfs}{rsfs}
\DeclareRobustCommand*\rsfs{%
\@fontswitch\relax\mathrsfs}

\theoremstyle{plain}
\newtheorem{thm}{Theorem}[section]
\newtheorem*{thm*}{Theorem}
\newtheorem{prop}[thm]{Proposition}

\newtheorem{lem}[thm]{Lemma}

\newtheorem{defi}[thm]{Definition}
\newtheorem{rmk}[thm]{Remark}

\newtheorem*{cor*}{Corollary}

\newtheorem{prop-defi}[thm]{Proposition-Definition}
\newtheorem{thm-defi}[thm]{Theorem-Definition}
\newtheorem{lem-defi}[thm]{Lemma-Definition}

\newtheorem*{question*}{Question}

\newtheorem{setup-def}[thm]{Setup-Definition}

\newdimen\argwidth
\def\db[#1\db]{
 \setbox0=\hbox{$#1$}\argwidth=\wd0
 \setbox0=\hbox{$\left[\box0\right]$}
  \advance\argwidth by -\wd0
 \left[\kern.3\argwidth\box0 \kern.3\argwidth\right]}

\newcommand{\cC}{\mathcal{C}}

\newcommand{\oO}{\mathcal{O}}
\newcommand{\pP}{\mathcal{P}}

\newcommand{\sS}{\mathcal{S}}

\newcommand{\uU}{\mathcal{U}}

\newcommand{\xX}{\mathcal{X}}

\newcommand{\Ob}{\mathcal{O}b}

\newcommand{\bL}{\mathbb{L}}
\newcommand{\bQ}{\mathbb{Q}}
\newcommand{\bR}{\mathbb{R}}

\renewcommand{\tilde}{\widetilde}
\renewcommand{\hat}{\widehat}

\newcommand{\dspec}{\mathop{{\textbf{Spec}}}\nolimits}

\newcommand{\lr}{\longrightarrow}

\newcommand{\ch}{\mathop{\rm ch}\nolimits}
\newcommand{\rk}{\mathop{\rm rk}\nolimits}
\newcommand{\td}{\mathop{\rm td}\nolimits}

\newcommand{\Spec}{\mathop{\rm Spec}\nolimits}

\newcommand{\Coh}{\mathop{\rm Coh}\nolimits}

\newcommand{\im}{\mathop{\rm im}\nolimits}

\newcommand{\bC}{\mathbb{C}}
\newcommand{\bZ}{\mathbb{Z}}

\newcommand{\bT}{\mathbb{T}}
\newcommand{\bI}{\mathbb{I}}
\newcommand{\bF}{\mathbb{F}}

\def\vir{\mathrm{\vir}}
\def\loc{\mathrm{\loc}}

\def\undxX{\underline \xX}
\def\Cinf{\cC^\infty}
\def\dd{\mathrm{d}}
\def\bT{\mathbb{T}}

\def\virt{^{\mathrm{vir}}}

\def\beq{\begin{equation}}
\def\eeq{\end{equation}}

\def\loc{{\mathrm{loc}}}

\makeatletter
 
  \@addtoreset{equation}{section}
\makeatother

\makeatletter
\def\@tocline#1#2#3#4#5#6#7{\relax
  \ifnum #1>\c@tocdepth 
  \else
    \par \addpenalty\@secpenalty\addvspace{#2}%
    \begingroup \hyphenpenalty\@M
    \@ifempty{#4}{%
      \@tempdima\csname r@tocindent\number#1\endcsname\relax
    }{%
      \@tempdima#4\relax
    }%
    \parindent\z@ \leftskip#3\relax \advance\leftskip\@tempdima\relax
    \rightskip\@pnumwidth plus4em \parfillskip-\@pnumwidth
    #5\leavevmode\hskip-\@tempdima
      \ifcase #1
       \or\or \hskip 1em \or \hskip 2em \else \hskip 3em \fi%
      #6\nobreak\relax
    \hfill\hbox to\@pnumwidth{\@tocpagenum{#7}}\par
    \nobreak
    \endgroup
  \fi}
\makeatother

\title[Cosection Localized Virtual Classes, Revisited]{Cosection Localization for D-Manifolds and $(-2)$-Shifted Symplectic Derived Schemes, Revisited}

\author{Michail Savvas}
\address{Department of Mathematics, The University of Texas at Austin, Austin, TX 78712, USA}
\email{msavvas@utexas.edu}

\begin{document}

\maketitle

\begin{abstract}
This is a continuation of prior work of the author on cosection localization for d-manifolds. We construct reduced virtual fundamental classes for derived manifolds with surjective cosections and cosection localized virtual fundamental classes for $(-2)$-shifted symplectic derived schemes in larger generality. Moreover, using recent results of Oh--Thomas, we show that the algebraic and differential geometric constructions of reduced and cosection localized virtual fundamental classes of $(-2)$-shifted symplectic derived schemes yield the same result in homology. We obtain applications towards the construction and integrality of reduced invariants in Donaldson--Thomas theory of Calabi--Yau fourfolds.
\end{abstract}

\setcounter{tocdepth}{1}
\tableofcontents
\setcounter{tocdepth}{2}

\section{Introduction}

Coherent sheaves and vector bundles are objects of fundamental importance in algebraic geometry. From an enumerative perspective, the construction and study of invariants counting sheaves on algebraic varieties satisfying given constraints plays a foundational role with a prominent case of interest in theoretical physics and string theory being sheaf-theoretic enumerative invariants of Calabi--Yau manifolds, commonly referred to as Donaldson or Donaldson--Thomas invariants.

Let $W$ be a smooth, projective complex Calabi-Yau manifold and $M$ a proper moduli scheme parameterizing stable sheaves on $W$. While Donaldson--Thomas invariants were defined in the '90s \cite{Thomas} when $\dim W =3$ by integrating cohomology classes on $M$ against a virtual fundamental cycle $[M]\virt$ \cite{BehFan,LiTian} in the homology or Chow theory of $M$, the fourfold case required a different strategy to produce a virtual cycle which took a while to develop. 

In \cite{CaoLeungDT4}, the authors introduced a gauge-theoretic approach which yields invariants in certain cases. Subsequently, Borisov--Joyce \cite{BorisovJoyce} used the existence of a $(-2)$-shifted derived enhancement of $M$ \cite{PTVV} together with an orientation \cite{CaoGrossJoyceOrient} to carry out an intricate truncation procedure that produces a compact, oriented d-manifold in the sense of \cite{JoyceDMan}, which admits a well-defined virtual fundamental class $[M]\virt \in H_{\ast} (M,\bZ)$. This gave access to the desired virtual fundamental class, not of algebraic, but rather differential geometric flavor.

The algebraic counterpart $[M]\virt \in A_\ast(M, \bZ[1/2])$ was constructed by Oh--Thomas \cite{OhThomasI} using the derived enhancement of $M$ as well as a localized version of the square root Euler class of Edidin--Graham \cite{EdidinGrahamQuadric}. In \cite{OhThomasII}, it was further proved that the differential geometric and algebraic virtual class are compatible with each other.
\medskip

Donaldson--Thomas theory of Calabi--Yau fourfolds (DT4 theory) is an area of intense ongoing research activity in recent years (cf. \cite{BaeKoolParkI, CaoKoolDT4PT4, CaoKoolMonavariKDT4PT4, CaoKoolMonavariPT4LocalCY, CaoTodaGVDescendent, CaoTodaDerivedDT4, CaoGV0Fano, CaoMaulikToda2, CaoMaulikToda, COT2, COT1} for a partial list of developments). 

One of the main tools used traditionally in the construction and study of enumerative invariants is localization of virtual cycles by cosection, introduced by Kiem--Li \cite{KiemLiCosection}. In the setting of DT4 theory, the differential geometric version of cosection localization \`{a} la Borisov--Joyce was developed in \cite{SavCosection}, while the algebraic version \`{a} la Oh--Thomas was established in \cite{KiemPark}. 

Building on the results of \cite{OhThomasII}, the purpose of this paper is to complement the results of \cite{SavCosection} by establishing stronger, more general statements and simplifying their proofs and to prove the compatibility of these two approaches. Namely, we do the following:
\begin{enumerate}
\item Construct reduced virtual fundamental classes for d-manifolds with surjective cosections. Given the utility of reduced invariants in \cite{COT2,COT1,BaeKoolParkI}, this provides access to integral homological versions of these invariants.
\item Define integral homological cosection localized and reduced virtual fundamental classes for $(-2)$-shifted symplectic derived schemes in full generality, generalizing the results of \cite{SavCosection}.
\item Prove that the differential geometric and algebraic versions of these virtual classes agree whenever they are both defined.
\end{enumerate}

\subsection*{Statement of results} We now give summarized statements of the main results of the paper for the convenience of the reader. For full details, we refer to the main text.

We begin with reduced virtual fundamental classes of d-manifolds with surjective cosections.

\begin{thm*} [Theorem~\ref{thm:reduced class for d-manifold}, Subsection~\ref{subsection: red class for mult cosection d-manifold}]
Let $\xX$ be a compact, oriented d-manifold of virtual dimension $n$ with underlying topological space $X$, equipped with a surjective morphism $\underline{\sigma} \colon h^1(\bT_{\xX}) \to \bR_X^k$. Then $[\xX]\virt = 0 \in H_n(X,\bZ)$ and there exists a well-defined reduced virtual fundamental class $[\xX]\virt_{\mathrm{red}} \in H_{n+k}(X,\bZ)$. 
\end{thm*}

We move on to $(-2)$-shifted symplectic derived algebraic geometry.

\begin{thm*} [Theorem~\ref{thm:cosection localization for derived schemes any cosection}, Theorem~\ref{thm:reduced for shifted symplectic}]
Let $(\undxX, \omega_{\undxX})$ be a projective, oriented $(-2)$-shifted symplectic derived scheme of virtual dimension $n = \rk \bL_{\undxX}|_X$, equipped with a morphism $\sigma \colon \bT_{\undxX}|_X[1] \to \oO_X$, where $X$ denotes the underlying classical scheme.
\begin{enumerate}
\item If $h^0(\sigma)$ fails to be surjective or is undefined on the closed subset $i \colon X(\sigma) \to X$ and is isotropic or non-degenerate with respect to the symplectic form $\omega_{\undxX}$, there exists a localized virtual fundamental class $[\undxX]\virt_{\mathrm{loc},\sigma} \in H_n(X(\sigma), \bZ)$ satisfying $i_\ast [\undxX]\virt_{\mathrm{loc},\sigma} = [\undxX]\virt \in H_n(X,\bZ)$.
\item If $\sigma$ is globally defined, surjective and isotropic with respect to the symplectic form $\omega_{\undxX}$, then $[\undxX]\virt = 0 \in H_n(X,\bZ)$ and there exists a reduced virtual fundamental class $[\undxX]\virt_\mathrm{red} \in H_{n+2}(X,\bZ)$.
\item If $\sigma$ is globally defined, surjective and non-degenerate with respect to the symplectic form $\omega_{\undxX}$, then $[\undxX]\virt = 0 \in H_n(X,\bZ)$ and there exists a reduced virtual fundamental class $[\undxX]\virt_\mathrm{red} \in H_{n+1}(X,\bZ)$.
\end{enumerate}
\end{thm*}

Using the machinery of complex Kuranishi charts of \cite{OhThomasII}, we are able to show that the differential geometric and algebraic localized or reduced virtual classes are equal in the appropriate sense.

\begin{thm*} [Theorem~\ref{thm:main thm}, Theorem~\ref{thm:isotropic surj cosection}, Theorem~\ref{thm:red for non-degenerate surj cosection}] Let $(\undxX, \omega_{\undxX})$ be a projective, oriented $(-2)$-shifted symplectic derived scheme, equipped with a morphism $\sigma \colon \bT_{\undxX}|_X[1] \to \oO_X$. 

Then, in each of the three scenarios (1)-(3) of the previous theorem, the image of the respective algebraic virtual fundamental cycle, constructed in \cite{KiemPark, BaeKoolParkI}, under the cycle class map $A_\ast(X, \bZ[1/2]) \to H_{2\ast}(X, \bZ[1/2])$ is equal to the image of the homological virtual fundamental class under the natural map $H_\ast(X,\bZ) \to H_\ast(X,\bZ[1/2])$.
\end{thm*}

In Section~\ref{sec:applications}, we apply these results to construct reduced invariants and obtain integrality properties for surface counting invariants and cosection localized virtual classes of Calabi--Yau fourfolds, as well as Gopakumar--Vafa type invariants of holomorphic symplectic fourfolds.

\subsection*{Organization of the paper} In \S2, we recall useful background on d-manifolds, $(-2)$-shifted symplectic derived schemes, their local properties and virtual classes. \S3 is devoted to reduced virtual fundamental classes of d-manifolds, while \S4 treats cosection localized and reduced virtual fundamental classes of $(-2)$-shifted symplectic derived schemes. In \S5, we prove the compatibility between the differential geometric and algebraic virtual classes in all cases. Finally, \S6 explores applications of our results to virtual classes and enumerative invariants in DT4 theory.

\subsection*{Acknowledgements} The author is indebted to Young--Hoon Kiem for prior collaborations and many enlightening conversations on cosection localization and other topics over the years. We would also like to greatly thank Hyeonjun Park and Martijn Kool for their interest, helpful discussions and motivating questions that encouraged the writing of this paper. Finally, we are grateful to the organizers and participants of the Workshop on moduli spaces, virtual invariants and shifted symplectic structures held at KIAS in June 2023 for providing a stimulating environment for discussions related to the topics of the paper.

\subsection*{Notation and conventions} 
We work with Borel-Moore homology and singular (co)homology with integer coefficients. These are equal for compact spaces and afford us the requisite flexibility to work with fundamental classes of (non)compact spaces. We refer the reader to \cite{Bredon, Iversen} for a comprehensive treatment. To avoid possible confusion, we always specify the coefficients we use for homology groups or Chow groups.

We work with the theory of $\Cinf$-schemes and derived manifolds developed by Joyce \cite{JoyceDMan}. Our $\Cinf$-schemes are separated, second countable and locally fair, meaning that their underlying topological space is Hausdorff, second countable and they are locally modelled by the spectrum of a finitely generated $\Cinf$-ring. By definition, such $\Cinf$-schemes can be the underlying scheme of a d-manifold. Unless otherwise stated, the underlying topological spaces of $\Cinf$-schemes and d-manifolds are assumed to be Euclidean Neighbourhood Retracts (ENR) for simplicity. This is automatically the case for complex analytic spaces or algebraic schemes. For the purpose of working with virtual classes and cosections, this assumption may be removed and is thus not restrictive (see \cite[Section~5.2]{SavCosection}).

All algebraic and derived schemes are defined over the complex numbers $\bC$. Classical and derived algebraic schemes and complex analytic spaces are assumed separated and of finite type. The reader can consult \cite{Toen} for general background on derived algebraic geometry and shifted symplectic structures.

$\undxX$ will typically denote a derived scheme and $\omega_{\undxX}$ a $(-2)$-shifted symplectic form on $\undxX$. $\xX$ will typically denote a d-manifold and $X_{\Cinf}$ its underlying $\Cinf$-scheme. $X$ will always be used for the underlying topological space of $\undxX$ or $\xX$ and sometimes, by abuse of notation, the classical truncation of $\undxX$ as well. Sometimes, $X$ will also denote the complex analytic space associated to the classical truncation $X$ of $\undxX$. These distinctions will be explicitly stated or should be clear from context throughout the paper. In general, we will refer to the notation $X_{\Cinf}$ when we need to use the existence of quasicoherent sheaves on a $\Cinf$-scheme, whereas we only consider homology groups of the underlying topological spaces.

\section{Some background}

In this section, we collect some requisite background and introduce terminology that will be used repeatedly in the rest of the paper. A slightly more detailed account of some overlapping material is given in \cite[Section~2]{SavCosection}.

\subsection{$\Cinf$-schemes and d-manifolds}

The main references for the theory
of $\Cinf$-schemes and d-manifolds we use are the books \cite{JoyceCoo} and \cite{JoyceDMan} by Joyce.

Roughly speaking, under the conventions of this paper, a $\Cinf$-scheme is a
locally ringed space that is locally isomorphic to the (real) spectrum of the
$\Cinf$-ring $\Cinf(\bR^n)/I$, where $I$ is an ideal of smooth functions on $\bR^n$. $\Cinf$-schemes admit a theory of (quasi)coherent sheaves and vector bundles, which largely mirrors that of sheaves on algebraic schemes.

Building on $\Cinf$-geometry, d-manifolds form a theory of derived differential geometry which is analogous to that of quasi-smooth derived or dg-schemes. Their local structure is described by \cite[Example~1.4.4]{JoyceDMan} as follows.

A principal d-manifold $\xX = \sS_{Y,E,s}$
is determined by local Kuranishi charts, i.e. the data $(Y,E,s)$ of a smooth manifold $Y$, a smooth vector bundle $E$ on $Y$ and a smooth section $s \in \Cinf(E)$: The underlying $\Cinf$-scheme $X_{\Cinf}$ of $\xX$ is the affine scheme obtained as the spectrum of the $\Cinf$-ring $\Cinf(Y)/I$, where $I = (s)$ is the ideal generated by the components of the section $s$. The higher data are expressed by the morphism $E^\vee |_{X_{\Cinf}} \xrightarrow{s^\vee} I/I^2$ of $\oO_{X_{\Cinf}}$-modules, so that the underlying topological space $X$ is the zero locus of the section $s$. Moreover, the tangent complex
(or virtual tangent bundle) of $X$ is the two-term complex
\begin{align} \label{eq:tangent complex of d-manifold}
\bT_\xX = [T_Y|_{X_{\Cinf}} \xrightarrow{\dd s} E|_{X_{\Cinf}}].
\end{align}
where for any choice of connection $\nabla$ on $E^\vee$, $\dd s$ is the restriction of $\nabla s \colon T_Y \to E$ to $X_{\Cinf}$.

In general, it follows by definition that every d-manifold $\xX$ is locally equivalent to a principal d-manifold. When $\xX$ is compact (i.e. its underlying topological space $X$ is
compact), Joyce shows that this is true globally.

\begin{thm-defi}\cite[Theorems~4.29, 4.34]{JoyceDMan} \label{thm-defi:global presentation of d-manifold} Let $\xX$ be a compact d-manifold. Then $\xX \simeq \sS_{Y,E,s}$ where $Y$ is an open subset of a Euclidean space $\bR^N$. We refer to this equivalence as a presentation of $\xX$ as a principal d-manifold or a global Kuranishi chart for $\xX$ (or $X$ when $\xX$ is clear from context). The virtual dimension of $\xX$ is then equal to $\mathrm{vd} = N-\rk E$.
\end{thm-defi}

The obstruction sheaf of $\xX$ is the coherent sheaf $\Ob_\xX := h^1(\bT_\xX) \in \Coh(X_{\Cinf})$.

A real cosection on $\xX$ is a morphism $\sigma \colon \Ob_\xX \to \bR_X$ of coherent sheaves on $X_{\Cinf}$, where $\bR_X$ denotes the trivial real line bundle on $X$. Analogously, a complex cosection is a morphism $\sigma \colon \Ob_\xX \to \bC_X = \bR_X^2$.

Finally, there is a notion of orientation for a d-manifold
$\xX$, which in the case of a principal d-manifold $\sS_{Y,E,s}$ roughly amounts to
an orientation of the determinant line bundle $\det \bT_\xX$ of the tangent complex~\eqref{eq:tangent complex of d-manifold}. As in the case of
manifolds, compact, oriented d-manifolds admit virtual fundamental classes of the expected dimension.

\begin{thm} \cite[Section~13]{JoyceDMan} If $\xX$ is a compact, oriented d-manifold
with underlying topological space $X$ and virtual dimension $\mathrm{vd}$, there exists a well-defined virtual fundamental
class $[\xX]\virt \in H_{\mathrm{vd}}(X,\bZ)$, which only depends on the bordism
class of $\xX$.
\end{thm}

By \cite{SavCosection}, we can give a simple description of $[\xX]\virt$ using a global Kuranishi chart $\xX \simeq \sS_{Y,E,s}$. Namely, following \cite{Siebert}, we can associate a (topological) normal cone $C(s) \subseteq E|_X$ to the section $s$ together with a homological fundamental class $[C(s)] \in H_N(C(s))$. Then, intersecting with the zero section of the oriented bundle $E|_X$ yields the virtual fundamental class so that $[\xX]\virt = 0_{E|_X}^![C(s)] \in H_{\mathrm{vd}}(X, \bZ)$.

\subsection{$(-2)$-shifted symplectic derived schemes} For a good overview of derived algebraic geometry, we refer the reader to \cite{Toen}, while for shifted symplectic structures the reader can consult the original paper \cite{PTVV}.

Roughly speaking, a derived scheme $\undxX$ is locally modelled by the (derived) spectrum $\dspec A$ where $A$ is a commutative differential (negatively) graded $\bC$-algebra. A $(-2)$-shifted symplectic structure on $\undxX$ is given by a 2-form
\begin{equation*}
    \omega_{\undxX} \colon \bT_{\undxX} \wedge \bT_{\undxX} \lr \oO_{\undxX}[-2]
\end{equation*}
which is non-degenerate and comes together with extra data that make it closed.

A $(-2)$-shifted symplectic scheme is a pair $(\undxX, \omega_{\undxX})$ consisting of a derived scheme $\undxX$ and a symplectic structure on $\undxX$. $(-2)$-shifted symplectic schemes $(\undxX, \omega_{\undxX})$ also admit an appropriate notion of orientation on a  (cf. \cite{BorisovJoyce}). 

The obstruction sheaf of $\undxX$ is defined to be the sheaf $\Ob_{\undxX} = h^1(\mathbb{T}_{\undxX} |_{X})$, while the obstruction space at a point $x \in \underline{\xX}$ is the vector space $\Ob(\underline{\xX}, x) = h^1(\mathbb{T}_{\underline{\xX}}|_x)$. A cosection is a morphism $\sigma \colon \bT_{\undxX}|_X[1] \to \oO_X$. We will often work with the induced map on cohomology $h^0(\sigma) \colon \Ob_{\undxX} \to \oO_X$.

In particular, the symplectic form $\omega_{\underline{\xX}}$ induces a family of non-degenerate quadratic forms $q_x \colon \Ob(\undxX, x) \to \bC$.

In \cite{JoyceSch}, the local structure of $(\undxX, \omega_{\undxX})$ around every point $x \in X$ is described by Kuranishi charts in Darboux form as follows. Even though we won't need this, we do mention the important fact that different Kuranishi charts can be glued in the appropriate, homotopic sense.

\begin{thm-defi} \emph{\cite[Example~5.16]{JoyceSch}} \label{Darboux chart}
Let $x \in \undxX$ be a point of a $(-2)$-shifted symplectic derived scheme $(\undxX, \omega_{\undxX})$. Then there exists a Zariski open neighbourhood $\underline{U} = \dspec A \to \undxX$ around $x$ such that:
\begin{enumerate}
    \item $(A, \delta)$ is a commutative differential graded algebra, with degree zero part $A^0$ a smooth $\bC$-algebra of dimension $N$, with a set of \'{e}tale coordinates $\{ x_i \}_{i=1}^N$. 
    \item As a graded algebra, $A$ is freely generated over $A^0$ by variables $\{ y_j \}_{j=1}^r$ and $\{ z_i \}_{i=1}^N$ in degrees $-1$ and $-2$ respectively.
    \item Letting $\omega_A$ be the pullback of $\omega_{\undxX}$ to $\dspec A$, there are invertible elements $q_1, ..., q_r \in A^0$ such that
    \begin{equation*}
        \omega_A = \sum_{i=1}^N \dd z_i \dd x_i + \sum_{j=1}^r \dd (q_j y_j) \dd y_j .
    \end{equation*}
    \item The differential $\delta$ of $A$ is determined by the equations
    \begin{align*}
        \delta x_i = 0, \quad \delta y_j = s_j, \quad \delta z_i = \sum_{j=1}^r y_j \left( 2q_j \frac{\partial s_j}{\partial x_i} + s_j \frac{\partial q_j}{\partial x_i} \right)
    \end{align*}
    where the elements $s_j \in A^0$ satisfy
    \begin{equation*}
        q_1 s_1^2 + ... + q_r s_r^2 = 0.
    \end{equation*}
\end{enumerate}

We write $Y = \Spec A^0$, $E$ for the trivial vector bundle of rank $r$ whose dual has basis given by the variables $y_j$ and $F$ for the trivial vector bundle whose dual has basis given by the variables $z_i$. We refer to all of the above data as a \emph{derived algebraic Darboux chart} for $\undxX$ at $x$.

It follows that the classical truncation $U = h^0(\underline{U}) \subseteq X$ is equipped with the following data:
\begin{enumerate}
    \item A smooth affine scheme $Y$ of dimension $N$.
    \item A vector bundle $E$ on $Y$ of rank $r$ with a non-degenerate quadratic form $q$, given by
    \begin{align*}
        q_u(y_1, ..., y_r) = q_1(u) y_1^2 + ... + q_r(u) y_r^2
    \end{align*}
    on each fiber of $E$ over $u \in U$. Thus, $E$ is naturally an orthogonal bundle, i.e. a $O(r,\bC)$-bundle.
    \item A $q$-isotropic section $s \in \Gamma(E)$ whose scheme-theoretic zero locus is $U \subseteq V$.
    \item A three-term perfect complex of amplitude $[0,2]$
    \begin{align*}
        \mathbb{E}_\bullet := \bT_{\underline{U}}|_U = [T_Y|_U \xrightarrow{ds} E|_U \xrightarrow{t} F|_U].
    \end{align*}
    such that $q_u$ gives a non-degenerate quadratic form on each obstruction space $\Ob(\undxX, u)$. 
    
    We have a natural morphism $\phi \colon \mathbb{E}^\bullet = (\mathbb{E}_\bullet)^\vee \to \bL_U$, which is surjective in degree $-1$ and an isomorphism in degree $0$. $\mathbb{E}^\bullet$ is symmetric in the sense that $\mathbb{E}^\bullet \cong \mathbb{E}_\bullet[2]$ via $q$ and the morphism $\phi$ is called the induced (three-term) \emph{symmetric obstruction theory} on $U$. We write $\mathrm{vd} = \rk \mathbb{E}_\bullet$ for the virtual dimension of $U$. 
\end{enumerate}

We refer to the data $(Y,E,s)$ as an \emph{algebraic Darboux chart} or \emph{algebraic, orthogonal Kuranishi chart} for $X$ at $x$.

Working complex analytically, since the functions $q_j$ are non-zero at $x$, by possibly shrinking $V$ around $x$, they admit square roots $q_j = r_j^2$ and after re-parametrizing $y_j \mapsto \frac{y_j}{r_j}$ we may assume that $q_j(u) = 1$. We then refer to the data $(Y,E,s)$ as an \emph{analytic Darboux chart} or \emph{analytic, orthogonal Kuranishi chart} for $X$ at $x$.
\end{thm-defi}

\subsection{Virtual fundamental classes}  \label{subsection: BJ and OT classes} Let $(\undxX, \omega_{\undxX})$ be a proper $(-2)$-shifted symplectic derived scheme. As in the previous subsection, let $\mathbb{E}^\bullet = \bL_{\undxX}|_X \to \bL_X$ be the induced (three-term) symmetric obstruction theory on $X$ with virtual dimension $\mathrm{vd} = \rk \mathbb{E}_\bullet$.

Suppose also that  $(\undxX, \omega_{\undxX})$ admits an orientation as defined in \cite{BorisovJoyce}, i.e. a trivialization of the determinant of $\mathbb{E}_\bullet$.

Under these assumptions, Borisov--Joyce \cite{BorisovJoyce} define a virtual fundamental class $[\undxX]\virt \in H_{\mathrm{vd}}(X,\bZ)$. Their method involves taking real smooth truncations of a cover of $\undxX$ by orthogonal Kuranishi charts and performing homotopical gluing to construct an associated compact, oriented d-manifold $\xX_{dm}$ of virtual dimension $\mathrm{vd}$. While $\xX_{dm}$ is not canonical, Borisov--Joyce show that it is well-defined up to bordism and hence so is its virtual fundamental class $[\xX_{dm}]\virt \in H_{\mathrm{vd}}(X,\bZ)$. Thus, one can unambiguously define $[\undxX]\virt := [\xX_{dm}] \in H_{\mathrm{vd}}(X,\bZ)$.

Now, suppose in addition that $\mathrm{vd} \in 2\bZ$ and $\mathbb{E}_\bullet$ admits a global resolution 
$$\mathbb{E}_\bullet \cong [B \xrightarrow{d} E \cong E^\vee \xrightarrow{d^\vee} B^\vee], $$
so that the intrinsic normal cone $\cC_X$ \cite{BehFan} of $X$ embeds as a closed substack of the vector bundle stack $[E/B]$. For example, this is the case when $X$ is projective by \cite[Section~4]{OhThomasI}. The choice of orientation of $(\undxX, \omega_{\undxX})$ then implies that $E$ is a special orthogonal bundle, i.e. an $SO(2r, \bC)$-bundle for some integer $r$. 

With these at hand, Oh--Thomas \cite{OhThomasI} use the fact that the cone $C = \cC_X \times_{[E/B]} E \subseteq E$ is isotropic and a localized version of the square root Euler class $\sqrt{e}$ of Edidin--Graham \cite{EdidinGrahamQuadric} to define an algebraic virtual fundamental cycle $[X]\virt = \sqrt{e}(E|_C, \tau)[C] \in A_{\mathrm{vd}/2}(X,\bZ[1/2])$. Here $\tau$ denotes the tautological section of the bundle $E|_C$, obtained by pullback along the projection map $C \to X$.

Assuming that $X$ is projective, in their follow-up work \cite{OhThomasII}, Oh--Thomas show that the image of the virtual cycle $[X]\virt$ under the cycle class map $A_{\mathrm{vd}/2}(X,\bZ[1/2]) \to H_{\mathrm{vd}}(X,\bZ[1/2])$ agrees with the image of $[\undxX]\virt$ under the natural map $H_{\mathrm{vd}}(X,\bZ) \to H_{\mathrm{vd}}(X,\bZ[1/2])$.

They achieve this by showing that there exists a \emph{global} orthogonal, complex Kuranishi chart $(Y,E,s)$ of $X$ in the sense of Theorem-Definition~\ref{Darboux chart}, which is suited to the $(-2)$-shifted symplectic structure of $\undxX$ and is analytic upon restriction on $X$ but smooth on the complement $Y \setminus X$. The following properties of the chart will be relevant for us:

\begin{enumerate}
\item $Y$ is a complex manifold, $E$ is an oriented smooth bundle over $Y$ and $s$ is a smooth section of $E$ such that $X \cong Z(s) \subseteq Y$.
\item $E$ admits a smooth, non-degenerate quadratic form $q \colon E \otimes E \to \bC_Y$ with respect to which $s$ is isotropic.
\item These data are compatible with the Darboux charts used by Borisov--Joyce \cite{BorisovJoyce}.
\item $E|_X$ is naturally a holomorphic vector bundle on the complex analytic space $X$, the derivative $\dd s \colon T_Y|_X \to E|_X$ is a holomorphic morphism of holomorphic bundles on $X$, and the shifted tangent complex $\bT_{\undxX}|_X [1]$ is given by
\begin{align}
\bT_{\undxX}|_X [1] \simeq [ T_Y|_X \xrightarrow{\dd s} E|_X \simeq E|_X^\vee \xrightarrow{\dd s^\vee} \Omega_Y|_X],
\end{align}
where the middle isomorphism is induced by the quadratic form $q$, which is also holomorphic when restricted to $X$.
\end{enumerate}

Since $E$ is a special orthogonal bundle, we may take $E \cong E^+ \otimes_\bR \bC = E^+ \oplus E^-$, a splitting into the real and imaginary parts of $E$, on which $\mathrm{Re} (q)$ is positive and negative definite respectively. $E^+$ inherits a natural orientation. Write $s^+$ for the projection of $s$ onto $E^+$ whose zero locus is $X \subseteq Y$, since $s$ is isotropic. The properties of the chart $(Y,E,s)$ then imply that we may take $\xX_{dm} = \sS_{Y,E^+,s^+}$ and we have the equality $$[X]\virt = [\xX_{dm}] \in H_{\mathrm{vd}}(X,\bZ[1/2]).$$

\section{Reduced virtual fundamental classes}

Let $\xX$ be a compact, oriented d-manifold of virtual dimension $n$ with underlying topological space $X$, whose truncation is the $\Cinf$-scheme $X_{\Cinf}$. Throughout this section, $\sigma \colon \Ob_{\xX} \to \bR_{X}$ will be a surjective real cosection for $\xX$. 

To fix more notation, by Theorem-Definition~\ref{thm-defi:global presentation of d-manifold}, $\xX$ is equivalent to a principal d-manifold $\sS_{Y,E,s}$, where $Y$ is an open subset of some Euclidean space $\bR^N$, $E$ is a smooth oriented vector bundle on $Y$ and $s$ a smooth section of $E$. Therefore, a model for the tangent complex of $\xX$ is given by~\eqref{eq:tangent complex of d-manifold}.

\subsection{Reduced classes via reduced d-manifold structure} We may assume that $Y$ retracts onto $X$ so that we have an isomorphism
\begin{align}
f_\ast \colon H_i(X) \lr H_i(Y)
\end{align}
is an isomorphism for any $i$, where $f \colon X \to Y$ denotes the inclusion map.
\smallskip

Given these data, we obtain a surjection $\sigma_E \colon E|_{X_{\Cinf}} \to \Ob_\xX \to \bR_X$ of sheaves on $X_{\Cinf}$. In particular, after possibly shrinking $Y$ around $X$, we may assume that $\sigma_E$ extends to a surjective morphism of vector bundles $\sigma_E \colon E \to \bR_Y$ on $Y$. Write $E(\sigma)$ for the kernel bundle of this map so that we have an exact sequence of vector bundles
\begin{align} \label{eq:loc 2.2}
0 \lr E(\sigma) \lr E \lr \bR_Y \lr 0.
\end{align}

Note that the orientation of $E$ induces a natural orientation on $E(\sigma)$.

After picking a smooth metric $g$ on $E$, we obtain a smooth splitting of~\eqref{eq:loc 2.2}
\begin{align}
 E \cong_{g} E(\sigma) \oplus \bR_Y.
\end{align}

Let $s(\sigma)_g$ be the first component of $s$ under this isomorphism. We obtain an induced principal d-manifold $\xX_g = \sS_{Y, E(\sigma), s(\sigma)_g}$ of virtual dimension $n+1$. 

Since $E(\sigma)$ is oriented, $\xX_g$ is an oriented d-manifold. Moreover, as $\xX$ is compact, after possibly shrinking $Y$ around $X$, we may assume that $E$ and $s$ extend to the closure $\overline{Y}$ of $Y$ in $\bR^N$ and, in addition, the vanishing locus of $s(\sigma)_g$ does not meet the boundary $\overline{Y} \setminus Y$. This implies that the underlying topological space $X_g$ of $\xX_g$ is compact, as a closed and bounded subset of $\overline{Y}$, which is disjoint from $\overline{Y} \setminus Y$. We conclude that $\xX_g$ is a compact, oriented d-manifold.

Write $f_g \colon X_g \to Y$ for the inclusion map. We may now give the following definition.

\begin{defi}
The \emph{reduced J-virtual fundamental class} of the d-manifold $\xX$ is defined as the J-virtual fundamental class of the d-manifold $\xX_g$, that is, 
\begin{align} \label{eq:reduced J-class}
[\xX]\virt_{\mathrm{red}, J} := (f_\ast)^{-1} (f_g)_\ast [\xX_g]\virt \in H_{n+1}(X).
\end{align}
\end{defi}

A priori, $[\xX]\virt_{\mathrm{red}, J}$ depends on the choices of the principal model $\sS_{Y,E,s}$ of $\xX$ as well as the metric $g$. We will see later on that it is independent of choices and hence well-defined.

\begin{rmk}
Given a choice of presentation $\xX \simeq \sS_{Y,E,s}$, it is easy to see that different choices of metric $g_1$ and $g_2$ result in bordant d-manifolds $\xX_{g_1}$ and $\xX_{g_2}$ over $Y$ by interpolating between the metrics via the path $g_t = (1-t) g_1 + t g_2$. Since the virtual fundamental class of a d-manifold only depends on its bordism class, it follows that $[\xX]\virt_{\mathrm{red}, J}$ is independent of the choice of metric and can only possibly depend on the choice of principal model $\sS_{Y,E,s}$.
\end{rmk}

\subsection{Reduced classes via cone reduction} We now give an alternative definition of the reduced class $[\xX]\virt_{\mathrm{red}}$. We keep using the notation of the previous subsection.

Recall that by the smooth cone reduction criterion \cite[Proposition~4.1]{SavCosection}, the normal cone $C(s)$ is contained entirely in the kernel bundle $E(\sigma)|_X$.

\begin{defi}
The \emph{reduced BF-virtual fundamental class} of the d-manifold $\xX$ is defined as
\begin{align} \label{eq:reduced BF-class}
[\xX]\virt_{\mathrm{red}, BF} := 0_{E(\sigma)|_X}^! [C(s)] \in H_{n+1}(X).
\end{align}
\end{defi}

Again, we shall soon see that $[\xX]\virt_{\mathrm{red}, BF}$ is independent of the choice of principal model $\sS_{Y,E,s}$ of $\xX$.

\subsection{Reduced classes are well-defined} In this subsection, we check that the definitions of reduced classes we have given are consistent with each other and independent of all choices involved, as promised.

\begin{prop}
$ [\xX]\virt_{\mathrm{red}, J} = [\xX]\virt_{\mathrm{red}, BF}$.
\end{prop}

\begin{proof}
By~\eqref{eq:reduced J-class} and~\eqref{eq:reduced BF-class}, we need to show that 
\begin{align*}
(f_g)_\ast [\xX_g]\virt = f_\ast 0_{E(\sigma)|_X}^! [C(s)] \in H_{n+1}(Y).
\end{align*}

Let $s'$ be a small perturbation of $s$ which intersects the zero section $Y$ of $E$ transversely. In particular, this implies that $s'(\sigma)_g$ is a small perturbation of $s(\sigma)_g$, which intersects the zero section of $E(\sigma)$ transversely.

Write $X'$ and $X_g'$ for the zero loci of $s'$ and $s'(\sigma)_g$ respectively. Let $f' \colon X' \to Y$ and $f'_g \colon X'_g \to Y$ be the inclusion maps.

By definition, we then have that 
\begin{align} \label{eq:loc 2.8}
(f_g)_\ast [\xX_g]\virt = (f'_g)_\ast [X'_g] \in H_{n+1}(Y).
\end{align}

By the cartesian square
\begin{align*}
\xymatrix{
E(\sigma)|_X \ar[d] \ar[r]^-{h} & E(\sigma) \ar[d] \\
X \ar[r]_-{f} & Y
}
\end{align*} 
we have $f_\ast 0_{E(\sigma)|_X}^! = 0_E^! h_\ast$ and hence 
\begin{align} \label{eq:loc 2.10}
f_\ast 0_{E(\sigma)|_X}^! [C(s)] = 0_{E(\sigma)}^! h_\ast [C(s)].
\end{align}

It is straightforward from the proof of the cone reduction criterion \cite[Proposition~5.1]{SavCosection} that $C(s) = C(s(\sigma)_g)$ as subsets of $E(\sigma)$. By the definition of the homological normal cone, it then follows that 
\begin{align} \label{eq:loc 2.11}
h_\ast[C(s)] = h_\ast [C(s(\sigma)_g)] = [\Gamma_{s(\sigma)_g}] = [\Gamma_{s'(\sigma)_g}] \in H_N(E(\sigma)),
\end{align}
where the last equality is due to the fact that $s'(\sigma)_g$ is a small perturbation of $s(\sigma)_g$.

Since $(f_g')_\ast [X_g'] = 0_{E(\sigma)}^! [\Gamma_{s'(\sigma)_g}]$, combining~\eqref{eq:loc 2.8}, \eqref{eq:loc 2.10} and \eqref{eq:loc 2.11} finishes the proof.
\end{proof}

It follows from the above proposition that $[\xX]\virt_{\mathrm{red},J}$ does not depend on the choice of metric $g$, since $[\xX]\virt_{\mathrm{red},BF}$ does not involve the metric. From now on, we write $[\xX]\virt_{\mathrm{red}}$ for the reduced virtual fundamental class of $\xX$. To show that it is well-defined, we thus only need to check that it is independent of the principal d-manifold model $\sS_{Y,E,s}$ we have used for $\xX$.

\begin{prop}
$[\xX]\virt_{\mathrm{red}}$ is independent of the choice of presentation $\xX \simeq \sS_{Y,E,s}$ and thus well-defined.
\end{prop}

\begin{proof}
Suppose that $\xX \simeq \sS_{Y_1, E_1, s_1} \simeq \sS_{Y_2, E_2, s_2}$. We work with the BF-virtual fundamental class in what follows.

Following the argument in the proof of \cite[Proposition~5.2]{SavCosection}, we may reduce to the case of an equivalence of the following standard form: There is a submersion $h \colon Y_1 \to Y_2$, a surjection of vector bundles $\hat{h} \colon E_1 \to h^\ast E_2$ such that $\hat{h} \circ s_1 = h^\ast s_2$ and a commutative diagram
\begin{align*}
\xymatrix{
T_{Y_1}|_X \ar[d]_-{\dd h} \ar[r]^-{\dd s_1} & E_1|_X \ar[d]^-{\hat{h}} \ar[dr] \\
h^\ast T_{Y_2}|_X  \ar[r]_-{\dd (h^\ast s_2)} & E_2|_X \ar[r] & \Ob_{\xX} \ar[r]_-{\sigma} & \bR_X,
}
\end{align*}
where the morphism of tangent complexes is a quasi-isomorphism on all fibers over points $x \in X$.

Since both derivatives $\dd s_1$ and $\dd s_2$ factor through the kernel bundles $E_1(\sigma)$ and $E_2(\sigma)$ we get an induced quasi-isomorphism (on fibers over $x \in X$) of two-term complexes
\begin{align*}
\xymatrix{
T_{Y_1}|_X \ar[d]_-{\dd h} \ar[r]^-{\dd s_1} & E_1(\sigma)|_X \ar[d]^-{\hat{h}} \\
h^\ast T_{Y_2}|_X  \ar[r]^-{\dd (h^\ast s_2)} & E_2(\sigma)|_X.
}
\end{align*}

It follows that the projection $\hat{h}$ induces a submersion $\hat{h}(\sigma) \colon E_1(\sigma) \to E_2(\sigma) $, which makes $E_1(\sigma)$ into a vector bundle over $E_2(\sigma)$. We pick a splitting of $\hat{h}$ so that $E_1(\sigma) \cong E_2(\sigma) \oplus E_3$, for some vector bundle $E_3$ on $Y$.

Now, $\hat{h}$ and hence $\hat{h}(\sigma)$ map $C(s_1)$ onto $C(s_2)$ such that $C(s_1)$ is a vector bundle over $C(s_2)$ and under the chosen splitting we actually have $C(s_1) \cong C(s_2) \oplus E_3|_X$.

Therefore, we conclude that
\begin{align*}
0_{E_1(\sigma)|_X}^! [C(s_1)] = 0_{E_2(\sigma)|_X \oplus E_3|_X}^! [C(s_2)] \oplus [E_3|_X] = 0_{E_2(\sigma)|_X}^! [C(s_2)],
\end{align*}
which is what we want.
\end{proof}

We may summarize as follows.

\begin{thm} \label{thm:reduced class for d-manifold}
Let $\xX$ be a compact, oriented d-manifold of virtual dimension $n$ with underlying topological space $X$ and $\sigma \colon \Ob_\xX \to \bR_X$ a surjective real cosection. Then there exists an induced, well-defined reduced virtual fundamental class $[\xX]\virt_{\mathrm{red}} \in H_{n+1}(X,\bZ)$.
\end{thm}

\subsection{The case of multiple cosections} \label{subsection: red class for mult cosection d-manifold} We finally note that we can define reduced virtual fundamental classes in the presence of multiple cosections. 

Namely, let $\underline{\sigma} \colon \Ob_{\xX} \to \bR_X^k$ be a surjective morphism of sheaves on $X_{\Cinf}$. Write $E(\underline{\sigma})$ for the kernel of the surjection $E \to \Ob_{\xX} \xrightarrow{\ \underline{\sigma} \ } \bR_X^k$, which is a vector bundle with a naturally induced orientation. Then, we may define a reduced virtual fundamental class $[\xX]_{\mathrm{red}}\virt \in H_{n+k}(X)$, exactly as before, in two ways:
\begin{enumerate}
\item After picking a metric $g$ on $E$, we have an isomorphism $E \cong E(\underline{\sigma}) \oplus \bR_X^k$. Let $s(\underline{\sigma})_g$ be the first component of $s$ under this isomorphism. Hence, we obtain a compact, oriented d-manifold $\xX_g = \sS_{Y, E(\underline{\sigma}),s(\underline{\sigma})_g}$ of virtual dimension $n+k$ with underlying topological space $X_g$ and we define $$[\xX]\virt_{\mathrm{red}, J} := (f_\ast)^{-1} (f_g)_\ast [\xX_g]\virt \in H_{n+k}(X),$$ where $f_g \colon X_g \to Y$ denotes the inclusion map.

\item A straightforward generalization of the cone reduction criterion \cite[Proposition~4.1]{SavCosection} shows that the normal cone $C(s)$ is contained in $E(\underline{\sigma})$. We can thus define
$$[\xX]\virt_{\mathrm{red}, BF} := 0_{E(\underline{\sigma})|_X}^! [C(s)] \in H_{n+k}(X).$$
\end{enumerate}

All the previous results immediately generalize in this setting so that $[\xX]\virt_{\mathrm{red}, J} = [\xX]\virt_{\mathrm{red}, BF}$ and the reduced virtual fundamental class is independent of the choice of presentation $\xX \simeq \sS_{E,Y,s}$ and metric $g$.


\section{Cosection localization and reduced virtual fundamental classes for $(-2)$-shifted symplectic derived schemes}

Let $(\undxX, \omega_{\undxX})$ be a proper, oriented, $(-2)$-shifted derived scheme with underlying classical scheme (and, by abuse of notation, topological space) $X$, equipped with a meromorphic cosection, i.e. a morphism $\sigma \colon \bT_{\undxX}|_X[1] \dashrightarrow \oO_X$ defined on an open subscheme $\underline{\uU}(\sigma) \subseteq \undxX$ with underlying classical scheme (and topological space) $U(\sigma)$. 

Write $\xX_{dm}, \uU(\sigma)_{dm}$ for the associated d-manifolds to the $(-2)$-shifted symplectic schemes $(\undxX, \omega_{\undxX})$ and $(\underline{\uU}(\sigma), \omega_{\undxX}|_{\underline{\uU}(\sigma)})$ respectively. In \cite{SavCosection}, we gave non-degeneracy conditions under which $\sigma$ induces a surjective real cosection on $\underline{\uU}_{dm}$. This in turn allowed us to localize the virtual fundamental class $[\xX_{dm}]\virt$ to the degeneracy locus $X(\sigma) = X \setminus U(\sigma)$ by applying the methods of \cite[Appendix~A]{KiemLiCosection}. 

In this section, we use the machinery of complex Kuranishi structures introduced recently by Oh--Thomas \cite{OhThomasII} to establish a similar result for \emph{any} cosection. The global complex Kuranishi charts constructed by Oh--Thomas allow us to bypass any subtleties arising from possible different choices of local Darboux charts for $(\undxX, \omega_{\undxX})$ and simplify the arguments of \cite{SavCosection}.

In the algebraic setting, cosection localization for $(-2)$-shifted symplectic schemes with isotropic cosections was established by Kiem--Park in \cite{KiemPark}. Thus our results match the expectation that there should be a differential geometric analogue. Conversely, the case of non-degenerate cosections was subsequently treated algebraically in \cite{BaeKoolParkI}. Hence, we now have a complete picture from both points of view. We will see in the next section that all definitions and different constructions are consistent with each other.
\medskip

In addition to the above, we assume that $X$ is projective. Take a global orthogonal, complex Kuranishi chart $(Y,E,s)$ of $X$, satisfying the properties (1)-(4) listed in Subsection~\ref{subsection: BJ and OT classes}. In particular, truncate this chart to obtain a principal d-manifold $\sS_{Y, E^+, s^+}$, which is a valid choice of model for the compact, oriented d-manifold $\xX_{dm}$ associated to $(\undxX, \omega_{\undxX})$.

For a meromorphic cosection $\sigma$ on a $(-2)$-shifted symplectic scheme $(\undxX, \omega_{\undxX})$, we say that $\sigma$ is \emph{isotropic} if the induced quadratic form $q_x$ on $\Ob({\undxX},x)^\vee$ is zero on the image of $\sigma_x^\vee$ for all $x \in U(\sigma)$. Equivalently, the composition 
\begin{align} \label{eq:loc 3.3}
\oO_{U(\sigma)} \xrightarrow{\sigma^\vee} \bL_{\underline{\uU}(\sigma)}|_{U(\sigma)}[-1] \xrightarrow{\ \omega_{\undxX}|_{U(\sigma)}\ } \bT_{\underline{\uU}(\sigma)}|_{U(\sigma)}[1] \xrightarrow{\ \sigma\ } \oO_{U(\sigma)}
\end{align}
vanishes.

On the other hand, if the composition~\eqref{eq:loc 3.3} is non-zero for all $x \in U(\sigma)$, we say that $\sigma$ is \emph{non-degenerate}.

\begin{thm} \label{thm:surjective cosection on Xdm}
Let $(\undxX, \omega_{\undxX})$ be an oriented $(-2)$-shifted derived scheme with a regular, i.e. globally defined, surjective cosection $\sigma$. Suppose that we have a global orthogonal, complex Kuranishi chart $(Y,E,s)$ for $X$. Then, if $\sigma$ is isotropic, it induces a surjective complex cosection $\sigma^+$ on the d-manifold $\xX_{dm} = \sS_{Y,E^+,s^+}$. If $\sigma$ is non-degenerate, it induces a surjective real cosection $\sigma^+$ on $\xX_{dm}$.
\end{thm}

\begin{proof}
Without loss of generality, by a standard argument (cf. \cite[Lemma~7.4]{KiemPark}) we may assume that $\sigma$ is represented by an honest morphism of complexes
\begin{align}
\xymatrix{
T_Y|_X \ar[r]^-{\dd s} \ar[d] & E|_X \simeq E|_X^\vee \ar[r]^-{\dd s^\vee} \ar[d]^-{\sigma} &\Omega_Y|_X \ar[d] \\
0 \ar[r] & \oO_X \ar[r] & 0.
}
\end{align}
It follows that the map $\sigma\colon E|_X \to \oO_X$ must be surjective.

In the non-degenerate case, $\sigma^+$ is defined via the composition $E^+|_X \subseteq E|_X \xrightarrow{ \ \sigma\ } \bC_X \to \bR_X$, where the last surjection is to be determined shortly. It is evident that this induces a well-defined real cosection on $\xX_{dm}$. 

If $\sigma_x \colon E|_x \to \bC_X|_x = \bC$ satisfies that $\sigma_x^\vee$ defines a non-degenerate line in $E|_x^\vee$ with respect to the quadratic form $q|_x$, then, since $\sigma_x$ is $\bC$-linear, it is the complexification of a surjective map $\sigma^+_x \colon E^+|_x \to \bR$. We thus have a projection $\bC_X \to \bR_X$ defined on the locus of $x \in X$ where $\sigma_x$ is non-degenerate, such that the induced composition $E^+|_x \to \bR$ is surjective. 

On the other hand, if $\sigma_x \colon E|_x \to \bC_X|_x = \bC$ satisfies that $\sigma_x^\vee$ defines an isotropic line in $E|_x^\vee$ with respect to the quadratic form $q|_x$, then $E^+|_x \subseteq E|_x \to \bC$ is surjective. 
If not, then $\im(\sigma_x^\vee) \cap (E^-|_x)^\vee$ would be non-empty, which is impossible. This concludes the proof.
\end{proof}

As a corollary, we may now generalize \cite[Theorem~5.8]{SavCosection}. 

\begin{thm} \label{thm:cosection localization for derived schemes any cosection}
Let $(\undxX, \omega_{\undxX})$ be a projective, oriented $(-2)$-shifted symplectic derived scheme of virtual dimension $n = \rk \bL_{\undxX}|_X$ with a meromorphic, surjective cosection $\sigma$, which is isotropic or non-degenerate. Then there exists a cosection localized virtual fundamental class $[\xX_{dm}]\virt_{\loc,\sigma} \in H_n(X(\sigma),\bZ)$ satisfying
$$ i_\ast [\xX_{dm}]\virt_{\loc,\sigma} = [\xX_{dm}]\virt \in H_n(X, \bZ),$$
where $i\colon X(\sigma) \to X$ is the inclusion map.
\end{thm}

\begin{proof}
By Theorem~\ref{thm:surjective cosection on Xdm}, $\sigma$ induces a surjective cosection $\sigma^+$ on the open sub-d-manifold of $\xX_{dm}$ with underlying topological space $U(\sigma)$. The proof of \cite[Theorem~5.8]{SavCosection} then applies verbatim to show the statement, using \cite[Theorem-Definition~5.3, Remark~5.4]{SavCosection}.
\end{proof}

\begin{rmk}
When $X$ is proper and $\sigma$ is regular, i.e. defined everywhere on $X$, then $\sigma$ is either non-degenerate or isotropic depending on whether the morphism~\eqref{eq:loc 3.3} is non-zero or vanishes respectively. These are the two most common cases in the literature. The theorem above allows for the slightly greater flexibility of having a meromorphic cosection $\sigma$.
\end{rmk}

Theorem~\ref{thm:surjective cosection on Xdm}  and its proof immediately imply the following statement on reduced classes.

\begin{thm} \label{thm:reduced for shifted symplectic}
Let $(\undxX, \omega_{\undxX})$ be a projective, oriented $(-2)$-shifted symplectic derived scheme of virtual dimension $n$ with a regular, everywhere surjective cosection $\sigma$. Then:
\begin{enumerate}
\item If $\sigma$ is non-degenerate, it induces a surjective real cosection on $[\xX_{dm}]$. Hence, there is a reduced virtual fundamental class $[\xX_{dm}]\virt_{\mathrm{red}} \in H_{n+1}(X,\bZ)$.
\item If $\sigma$ is isotropic, the morphism $E^+ \subseteq E \xrightarrow{\sigma} \bC_X$ induces a surjective complex cosection $\sigma^+ \colon \Ob_{\xX_{dm}} \to \bC_X = \bR_X^2$ on $[\xX_{dm}]$. Hence, there is a reduced virtual fundamental class $[\xX_{dm}]\virt_{\mathrm{red}} \in H_{n+2}(X,\bZ)$.
\end{enumerate}
\end{thm}

\section{Comparison between differential geometric and complex analytic cosection localized virtual fundamental classes}

We now proceed to extend the comparison result of \cite{OhThomasII} in the case of even virtual dimension, showing that the differential geometric and complex analytic versions of the virtual fundamental class of a $(-2)$-shifted symplectic scheme are equal in homology, to the cosection localized case.

\subsection{The square root Euler class} \label{subsection:square root euler} We start by giving a homological formula for the square root Euler class of Oh--Thomas \cite{OhThomasI} and Kiem--Park \cite{KiemPark}. 

We set up some notation. Let $E$ be an $SO(2r, \bC)$-bundle over a scheme $Y$ together with two isotropic sections $s,\sigma^\vee$ satisfying $s \cdot \sigma^\vee = 0$ where $(-) \cdot (-)$ denotes the pairing induced by the quadratic form on $E$. Using the quadratic structure of $E$, the section $\sigma^\vee$ is equivalent to a cosection $\sigma \colon E \cong E^\vee \to \oO_Y$. 

Let $X$ be the vanishing locus of $s$ and $X(\sigma) \subseteq X$ the vanishing locus of $\sigma|_X$ and moreover assume that $s,\sigma^\vee$ are linearly independent away from a closed subset $Z \subseteq Y$ which contains the vanishing locus $\sigma^{-1}(0)$ and satisfies $X \cap Z = X(\sigma)$ (cf. \cite[Definition~6.4]{KiemPark}). From now on, let us write $t=\sigma^\vee$ for brevity and consistency with the notation of Kiem--Park \cite{KiemPark}.

Kiem--Park then define a cosection localized square root Euler class map
\begin{align} \label{eq:loc 5.1}
\sqrt{e}(E,s;t) \colon A_\ast (Y, \bZ) \lr A_{\ast - r}(X(\sigma), \bZ [1/2]).
\end{align}

We establish some further notation and recall the definition in order to prove a formula in homology. 

Let $\rho \colon \tilde{Y} \to Y$ be the blowup of $Y$ along $X$ with exceptional divisor $D$. Write $L = \oO_{\tilde{Y}}(D)$ and $\tilde{E} = L^{\perp}/ L$ for the induced $SO(2r-2, \bC)$-bundle. Using the fact that $s \cdot t = 0$, the pullback $\rho^\ast t \in H^0(\tilde{Y}, \rho^\ast E)$ induces an isotropic section $\tilde{t}$ of $\tilde{E}$. Write $s_D$ for the tautological section of $L$.

We have a commutative diagram
\begin{align} \label{eq:loc 4.2}
\xymatrix{
0 \ar[r] & L \ar[r] \ar@{=}[d] & L^\perp \ar[r] \ar[d] & \tilde{E} \ar[r] \ar[d] & 0 \\
0 \ar[r] & L \ar[r] & \rho^\ast E \ar[r] & K \ar[r] & 0
}
\end{align}
where the rows are short exact sequences of vector bundles. By definition, $\rho^\ast s$ factors through the tautological section $s_D$.

We also have a commutative diagram
\begin{align}
\xymatrix{
D(t) \ar[r]^-{\imath''} \ar[d] & D(\tilde{t}) \ar[r]^-{\jmath'} \ar[d]^-{\rho''} & D \ar[r]^-{\jmath} \ar[d]^-{\rho'} & \tilde{Y} \ar[d]^-{\rho} \\
X(\sigma) \ar[r]_-{\imath''} & X(\sigma)^\# \ar[r]_{\imath'} & X \ar[r]_-{\imath} & Y,
}
\end{align}
in which $X(\sigma)^\# = \rho'(D \cap \tilde{t}^{-1}(0)) \cup X(\sigma) \subseteq X$ and the rightmost square is cartesian.

By our assumption on the independence of $s$ and $t$, \cite[Theorem~6.5]{KiemPark} shows that $X(\sigma)^\#_\mathrm{red} \subseteq X(\sigma)_\mathrm{red}$, so, while in general the codomain of $\sqrt{e}(E,s;t)$ is the Chow group of $X(\sigma)^\#$, we may replace it by $X(\sigma)$, as we have already done in~\eqref{eq:loc 5.1}.

Let $\gamma \in A_\ast (Y, \bZ)$. Then there exist  classes $\alpha \in A_\ast (\tilde{Y},\bZ), \beta \in A_\ast(X,\bZ)$ such that $\gamma = \rho_\ast \alpha + \imath_\ast \beta$ and Kiem--Park \cite{KiemPark} define
\begin{align} \label{eq: KiemPark def}
\sqrt{e}(E,s;t) = \rho_\ast'' \jmath^! \sqrt{e}(\tilde{E},\tilde{t}) \alpha + \imath''_\ast \sqrt{e}(F,t)\beta,
\end{align}
where $\sqrt{e}$ is the square root Euler class defined by Oh--Thomas \cite{OhThomasI}. It is shown in \cite{KiemPark} that this formula is well-defined and enjoys the expected properties of an Euler class.

Since $E$ is an oriented, orthogonal bundle, write $E \cong E^+ \otimes_{\bR} \bC$ and $s^+, t^+$ for the real parts of $s,t$ respectively. We reserve the same notation for other quadratic bundles and their sections, when clear from context.

\begin{lem} \label{lem:formula for square root}
Under the cycle class map $A_\ast(X(\sigma),\bZ [1/2]) \to H_{2\ast}(X(\sigma),\bZ [1/2])$, we have
\begin{align} \label{eq:square root formula}
\sqrt{e}(E,s;t) [Y] = [C(s^+)] \cap [\Gamma_{t^{+}}],
\end{align}
where the intersection takes place inside $E^+$.
\end{lem}

\begin{proof}
Using the notation preceding the lemma, since $[Y] = \rho_\ast [\tilde{Y}]$, \eqref{eq: KiemPark def} implies that
\begin{align} \label{eq: loc 4.5}
\sqrt{e}(E,s;t) (\alpha) & = \rho_\ast'' \jmath^! \sqrt{e}(\tilde{E},\tilde{t}) [\tilde{Y}] \\
\notag & = \rho_\ast'' \jmath^! ([C(\tilde{t}) ]\cap [\tilde{E}^{-}]) \\
\notag & = \rho_\ast'' \jmath^! ([C(\tilde{t}^+)] \cap [0_{\tilde{E}^{+}}])\\
\notag & = \rho_\ast'' \jmath^! ([\Gamma_{\tilde{t}^+}] \cap [0_{\tilde{E}^{+}}]),
\end{align}
where the two intermediate equalities follow from \cite[Theorem~10.2]{OhThomasII}, \cite[Theorem~7.23]{OhThomasII} and \cite[Equation~(7.25)]{OhThomasII}, and the last equality from the fact that, as homology classes on $\tilde{E}^+$, $[C(\tilde{t}^+)] = [\Gamma_{\tilde{t}^+}]$, whose intersection with the zero section is contained inside the zero locus $\tilde{t}^{-1}(0)$.

We may now choose a $\bC$-linear smooth splitting of the first row of~\eqref{eq:loc 4.2} and consider $\tilde{E}$ as a subbundle of $\rho^\ast E$, which intersects $L$ trivially. In particular, we can write $\rho^\ast E = \tilde{E} \oplus R$ for some vector bundle $R$, such that the inclusion map $L \to \rho^\ast E$ factors through $R$, and thus we have an induced splitting $\rho^\ast E^+ = \tilde{E}^+ \oplus R^+$. 

Note that since $L$ is an isotropic (complex) line bundle in $\rho^\ast E$, the composite morphism $L \to \rho^\ast E \to R^+$ is an isomorphism of smooth (real) bundles which maps the section $s_D$ to $\rho^\ast s^+$.

Based on these observations, it follows by the definition of $\jmath^!$ that
\begin{align} \label{eq: loc 4.6}
\jmath^!  ([\Gamma_{\tilde{t}^+}] \cap [0_{\tilde{E}^{+}}]) = [\Gamma_{\rho^\ast s^+}] \cap [\Gamma_{\tilde{t}^+}],
\end{align}
where the intersection takes place inside $\rho^\ast E^+$. 

Since the sections $\tilde{t}^+$ and $\rho^\ast t^+$ are homotopic to each other and their graphs intersect the graph of $\rho^\ast s^+$ at the same subset of $\rho^\ast E^+$, we further obtain
\begin{align} \label{eq: loc 4.7}
[\Gamma_{\rho^\ast s^+}] \cap [\Gamma_{\tilde{t}^+}] = [\Gamma_{\rho^\ast s^+}] \cap [\Gamma_{\rho^\ast{t}^+}].
\end{align}

Finally, observe that $[\Gamma_{\rho^\ast{t}^+}] = \rho^\ast [\Gamma_{t^+}]$, where by abuse of notation $\rho$ also denotes the map $\rho^\ast E^+ \to E^+$, so by the projection formula
\begin{align} \label{eq: loc 4.8}
\rho_\ast \left( [\Gamma_{\rho^\ast s^+}] \cap [\Gamma_{\rho^\ast{t}^+}] \right) & = \rho_\ast \left( [\Gamma_{\rho^\ast s^+}] \cap \rho^\ast [\Gamma_{t^+}] \right) \\
\notag & = [\Gamma_{s^+}] \cap [\Gamma_{{t}^+}] \\
\notag & = [C(s^+)] \cap [\Gamma_{t^+}],
\end{align}
again using the fact that $[\Gamma_{s^+}] = [C(s^+)]$, as homology classes on $E^+$.

Combining~\eqref{eq: loc 4.5}, \eqref{eq: loc 4.6}, \eqref{eq: loc 4.7} and \eqref{eq: loc 4.8} finishes the proof.
\end{proof}

\begin{rmk}
Observe that in the above, $E$, $s$ and $\sigma^\vee$ do not need to be algebraic or holomorphic, given that $X, Y$ are complex analytic spaces and the inclusions $X(\sigma) \to X$ and $X \to Y$ are complex analytic. It is sufficient that they are smooth or even continuous, and complex linear. This will be implicitly used in the following section.
\end{rmk}

\subsection{Everything agrees}
Let $(\undxX, \omega_{\undxX})$ be a projective, oriented $(-2)$-shifted symplectic derived scheme of (complex) virtual dimension $n = \rk \bL_{\undxX}|_X$ with an isotropic cosection $\sigma$, which is surjective on an open subset $U(\sigma)$ with complement $X(\sigma) = X \setminus U(\sigma)$. 

Take $(Y,E,s)$ a global complex Kuranishi chart for $X$ with $E$ an $SO(2r,\bC)$-bundle, $s$ isotropic so that $\sigma$ is represented by an honest morphism of complexes
\begin{align}
\xymatrix{
T_Y|_X \ar[r]^-{\dd s} \ar[d] & E|_X \simeq E|_X^\vee \ar[r]^-{\dd s^\vee} \ar[d]^-{\sigma} &\Omega_Y|_X \ar[d] \\
0 \ar[r] & \oO_X \ar[r] & 0,
}
\end{align}
and extends smoothly to a morphism of smooth vector bundles which is surjective on an open neighborhood $U(\sigma) \subseteq V$ of $U(\sigma)$ in $Y$. Moreover, using a smooth metric on $E$ and $|s|^2$ as a smooth function that vanishes on $X$ and is non-zero on the complement $Y\setminus X$, up to a smooth modification of $\sigma$ off $X$ and possible shrinking of $Y$ around $X$, we may assume that $\sigma \circ s = 0$ identically on $Y$, $\sigma^\vee$ is isotropic and $s, \sigma^\vee$ are independent away from $X(\sigma) = Y \cap \sigma^{-1}(0)$. Thus, we have in particular that $s \cdot \sigma^\vee = 0$ on $Y$.

Fix an isomorphism $E \cong E^+ \oplus E^-$, as usual and write $s^+, \sigma^+$ for the real components of $s, \sigma$ respectively.

\begin{lem} \label{lem:KP formula}
The Kiem--Park cosection localized virtual fundamental class of $X$ equals $\sqrt{e}(E,s;\sigma^\vee) [Y]$ in homology.
\end{lem}

\begin{proof}
Let $\tau$ be the tautological section of $E|_{C(s)}$ whose zero locus equals $X \subseteq C(s)$. By \cite[Theorem~8.2]{KiemPark}, the homology class of the algebraic cosection localized virtual fundamental class of $X$ is equal to
\begin{align*}
\sqrt{e}(E|_{C(s)},\tau;\sigma|_{C(s)}^\vee) [C(s)] \in H_n(X(\sigma), \bZ[1/2]).
\end{align*}

However, by deformation to the normal cone, the data $(Y, E, s)$  deform to the tautological section $(C(s), E|_{C(s)}, \tau)$, together with the cosection $\sigma$ which deforms to $\sigma|_{C(s)}$. Thus, by deformation invariance of the square root Euler class (see \cite[Sections~4.2, 4.3]{OhThomasI}), it follows that $$\sqrt{e}(E|_{C(s)},\tau;\sigma|_{C(s)}^\vee)[C(s)] = \sqrt{e}(E,s;\sigma^\vee) [Y],$$ as we want.
\end{proof}

\begin{thm} \label{thm:main thm}
$[\xX_{dm}]_{\loc,\sigma}\virt = \sqrt{e}(E,s;\sigma^\vee) [Y] \in H_n(X(\sigma), \bZ[1/2])$.
\end{thm}

\begin{proof}
This follows immediately from~\eqref{eq:square root formula} and Lemma~\ref{lem:KP formula}, since, by \cite[Theorem-Definition~5.3]{SavCosection}, it holds by construction that $[\xX_{dm}]_{\loc,\sigma}\virt = [C(s^+)] \cap [\Gamma_{(\sigma^+)^\vee}]$.
\end{proof}

Combining the above, we conclude that the algebraic and differential geometric cosection localized virtual fundamental classes agree in homology.

\begin{rmk}
Besides producing integral classes in homology, one advantage of the differential geometric approach to cosection localization is seen in the case of odd virtual dimension, where we still have potentially non-trivial cosection localized virtual classes, which however must be $2$-torsion by the results of \cite{OhThomasII}. It is an interesting question to examine whether this phenomenon of non-trivial $2$-torsion virtual fundamental classes (cosection localized or not) can indeed occur.
\end{rmk}

\subsection{Equality of reduced classes: the isotropic case}

We continue to explain that our comparison result is also true for reduced classes defined in the presence of surjective cosections.

We work in the following slightly more general setup (cf. \cite[Assumption~8.1]{KiemPark}): Let $X$ be a projective scheme with a symmetric obstruction theory $\mathbb{E}^\bullet \to \bL_X$, where $\mathbb{E}^\bullet$ is perfect of amplitude $[-2,0]$ and even rank $n = \rk \mathbb{E}^\bullet$ and there is a resolution $\mathbb{E}^\bullet \cong [B \xrightarrow{d} F \cong F^\vee \xrightarrow{d^\vee} B^\vee]$. Moreover, assume that we have an orientation $\oO_X \cong \det \mathbb{E}^\bullet$ so that $F$ is an $SO(2r, \bC)$-bundle and the cone $C = \cC_X \times_{[F/B]} F \subseteq F$ is isotropic. 

Let us now assume that there is a surjective, isotropic cosection $\sigma \colon \mathbb{E}_\bullet[1] = (\mathbb{E}^\bullet)^\vee[1] \to \oO_X$, which is induced by a morphism $\sigma \colon F \to \oO_X$.

By \cite{OhThomasII}, given the above, there is a global, orthogonal complex Kuranishi chart $(Y,E,s)$ for $X$ such that $F=E|_X, d=\dd s$, $s$ is an isotropic section of $E$ and $\sigma$ extends smoothly to a morphism $\sigma \colon E \to \bC_Y$. As in the previous subsection, we may assume that $\sigma^\vee$ is isotropic, $s \cdot \sigma^\vee = 0$ and $s, \sigma^\vee$ are independent away from $Y \cap \sigma^{-1}(0) = X(\sigma)$.

The associated compact, oriented d-manifold to these data is $\xX = \sS_{Y,E^+,s^+}$.
\smallskip

Consider the $SO(2r-2, \bC)$-bundle $\tilde{E} = \langle \sigma^\vee \rangle^\perp / \langle \sigma^\vee \rangle$. By the orthogonal cone reduction lemma \cite{KiemPark}, there is a closed embedding $C(s) \subseteq \tilde{E}$ and the reduced algebraic virtual fundamental cycle of $X$ is defined by the formula \cite[Definition~8.7]{KiemPark}
\begin{align}
[X]\virt_{\mathrm{red}} := \sqrt{e}(\tilde{E}|_{C(s)}, \tau)[C(s)] \in A_{\frac{1}{2}n+1}(X, \bZ[1/2]),
\end{align}
where $\tau$ denotes, as usual, the tautological section of $\tilde{E}|_{C(s)}$ whose zero locus is $X$.

By the orthogonality and independence of $s$ and $\sigma^\vee$, $s$ induces a smooth isotropic section $\tilde{s}$ of $\tilde{E}$ with zero locus $X$. We thus obtain a principal compact, oriented d-manifold $\xX_\mathrm{red} = \sS_{Y, \tilde{E}^+, \tilde{s}^+}$. 

It is easy to see that $C(s) = C(\tilde{s})$, so by deformation to the normal cone, we obtain the equality 
\begin{align} \label{eq:loc 4.12}
[X]_\mathrm{red}\virt = \sqrt{e}(\tilde{E}|_{C(s)}, \tau)[C(s)] = \sqrt{e}(\tilde{E}, \tilde{s})[Y] = [\xX_{\mathrm{red}}]\virt \in H_{n+2}(X, \bZ[1/2]).
\end{align}

On the other hand, we may choose an orthogonal $\bC$-linear smooth splitting $E \cong \tilde{E} \oplus R$. Using the splitting, the surjection of smooth vector bundles $R \to \bC_Y \cong \bR_Y^{2}$ restricts to a surjection $\sigma^+ \colon R^+ \subseteq E \to \bR_Y^2$. This follows from the fact that $\sigma^\vee$ is isotropic and factors through $R^\vee \cong R$, so its projection to $R^+$ is always non-zero. As in Theorem~\ref{thm:reduced for shifted symplectic}, $\sigma^+$ defines a surjective cosection of rank $2$ on the d-manifold $\xX = \sS_{Y,E^+,s^+}$. By Subsection~\ref{subsection: red class for mult cosection d-manifold}, the reduced virtual fundamental class of $\xX$ is then by definition
\begin{align} \label{eq:loc 4.13}
[\xX]_{\mathrm{red}}\virt := [\xX_{\mathrm{red}}]\virt \in H_{n+2}(X, \bZ).
\end{align}

Combining~\eqref{eq:loc 4.12} and \eqref{eq:loc 4.13} yields the following result.

\begin{thm} \label{thm:isotropic surj cosection}
Let $X$ be a projective scheme equipped with an oriented symmetric obstruction theory of amplitude $[-2,0]$, isotropic normal cone and isotropic surjective cosection $\sigma$, as above. Let $\xX$ be the compact, oriented d-manifold associated to these data and $\sigma^+$ the induced surjective cosection on $\xX$. 

Then the algebraic reduced virtual fundamental cycle of $X$ and the reduced virtual fundamental class of $\xX$ are equal in homology with $\bZ[1/2]$-coefficients, that is,
$$[X]_{\mathrm{red}}\virt = [\xX]_{\mathrm{red}}\virt \in H_{n + 2}(X, \bZ[1/2]).$$
\end{thm}

\begin{rmk}
Of course, the most natural scenario in which the conditions of the theorem arise is when $X$ is the classical truncation of an oriented $(-2)$-shifted symplectic scheme $(\undxX, \omega_{\undxX})$.
\end{rmk}

\subsection{Equality of reduced classes: the non-degenerate case}
We keep the notation of the previous subsection and move on to the case of a non-degenerate cosection $\underline{\sigma} \colon \mathbb{E}_\bullet[1] \to \oO_X^{\oplus d}.$

Following \cite{BaeKoolParkI,SavCosection}, non-degeneracy here means that the composition
$$\oO_X^{\oplus d} \xrightarrow{\underline{\sigma}^\vee} \mathbb{E}^\bullet [-1] \cong \mathbb{E}_\bullet[1] \xrightarrow{\underline{\sigma}} \oO_X^{\oplus d} $$
is an isomorphism. In particular, $\underline{\sigma}$ is surjective.

Let $\mathbb{E}^\bullet_\mathrm{red}$ be the cone of $\underline{\sigma}^\vee[1]$. Assume that $$(\mathbb{E}^\bullet_\mathrm{red})^\vee[2] \to \mathbb{E}_\bullet[2] \cong \mathbb{E}^\bullet \to \mathbb{E}^\bullet_\mathrm{red}$$ induces a symmetric form on $\mathbb{E}^\bullet_\mathrm{red}$ and there is a splitting of symmetric complexes $\mathbb{E}^\bullet \cong \mathbb{E}^\bullet_\mathrm{red} \oplus \oO_X^{\oplus d}[1]$.

In the notation of the previous subsection, we may fix a global, orthogonal complex Kuranishi chart $(Y,E,s)$ for $X$ and further assume that $\underline{\sigma}$ is induced by a split surjection of orthogonal bundles $F \to \oO_X^{\oplus d}$ which extends smoothly to a surjective $\bC$-linear morphism $\underline{\sigma} \colon E \to \bC_Y^d$ with kernel the oriented, orthogonal bundle $E_\mathrm{red}$. Since $\underline{\sigma}\circ s = 0$, $s$ factors through a section $s_\mathrm{red}$ of $E_\mathrm{red}$.

Let $\xX = \sS_{Y,E^+, s^+}$ and $\xX_\mathrm{red} = \sS_{Y,E_\mathrm{red}^+, s_\mathrm{red}^+}$ be the associated compact, oriented d-manifolds.

By \cite[Proposition~4.15, Theorem~4.5]{BaeKoolParkI}, the isotropic normal cone $C(s) = C(s_\mathrm{red})$ is contained in $E_\mathrm{red}$ and the reduced virtual fundamental cycle of $X$ is defined by the formula
\begin{align} \label{eq:loc 5.14}
[X]_\mathrm{red}\virt & := \sqrt{e}(E_\mathrm{red}|_{C(s)}, \tau) [C(s)] \\
\notag & = \sqrt{e}(E_\mathrm{red}|_{C(s_\mathrm{red})}, \tau)[C(s_\mathrm{red})] \\
\notag & = \sqrt{e} (E_\mathrm{red}, s_\mathrm{red})[Y] = [\xX_\mathrm{red}]\virt \in H_{n+d}(X, \bZ[1/2]),
\end{align}
where the last equality follows from Theorem~\ref{thm:main thm}.

As in Theorem~\ref{thm:reduced for shifted symplectic}, we have a composition $$\underline{\sigma}^+ \colon E^+ \subseteq E \xrightarrow{\underline{\sigma}} \bC_Y^d \to \bR_Y^d,$$ which defines a surjective cosection on the d-manifold $\xX$ and whose kernel bundle is exactly $E_{\mathrm{red}}^+$. By Subsection~\ref{subsection: red class for mult cosection d-manifold}, the reduced virtual fundamental class of $\xX$ is then by definition
\begin{align} \label{eq:loc 5.15}
[\xX]_{\mathrm{red}}\virt := [\xX_{\mathrm{red}}]\virt \in H_{n+d}(X, \bZ).
\end{align}

A similar argument to the previous subsection, combining~\eqref{eq:loc 4.12} and \eqref{eq:loc 4.13}, yields the following result.

\begin{thm} \label{thm:red for non-degenerate surj cosection}
Let $X$ be a projective scheme equipped with an oriented symmetric obstruction theory of amplitude $[-2,0]$, isotropic normal cone and non-degenerate cosection $\underline{\sigma}$, as above. Let $\xX$ be the compact, oriented d-manifold associated to these data and $\underline{\sigma}^+$ the induced surjective cosection on $\xX$. 

Then the algebraic reduced virtual fundamental cycle of $X$ and the reduced virtual fundamental class of $\xX$ are equal in homology, that is,
$$[X]_{\mathrm{red}}\virt = [\xX]_{\mathrm{red}}\virt \in H_{n + d}(X, \bZ[1/2]).$$
\end{thm}

\section{Applications: Integrality and reduced invariants} \label{sec:applications}

We give a few (related) immediate applications of our results regarding integrality and reduced invariants in Donaldson--Thomas theory of Calabi--Yau and, in particular, holomorphic symplectic and hyperk\"{a}hler fourfolds.

\subsection{Surface counting on Calabi--Yau fourfolds} Let $W$ be a Calabi-Yau fourfold, i.e. a smooth projective four-dimensional variety with trivial canonical bundle $\omega_W \colon \oO_W \cong K_W$.

In \cite{BaeKoolParkI}, the authors consider moduli spaces $P = P_v^{(q)}(W)$ of $\mathrm{PT}_q$-pairs $(F,s)$ satisfying
$$\ch(F) = (0, 0, \gamma, \beta, n-\gamma \cdot \td (W)) \in H^\ast(W,\bQ),$$
where $F$ is a coherent sheaf on $W$ and $s \in H^0(W,F)$. For details on the stability conditions involved, we refer the reader to \cite{BaeKoolParkI}.

By \cite[Theorem~1.4]{BaeKoolParkI}, $P$ naturally admits a $(-2)$-shifted symplectic derived enhancement $\underline{\pP}$ with a choice of orientation. Write $\mathbb{E}^\bullet \to \bL_P$ for the induced three-term symmetric obstruction theory on $P$.

Now, given a surface class $\gamma \in H^2(W, \Omega_W^2)$ and $q \in \lbrace -1,0,1 \rbrace$, let $\rho_\gamma$ be the codimension of the Hodge locus of $\gamma$. There is then an induced non-degenerate cosection $\underline{\sigma}_\gamma \colon \mathbb{E}_\bullet[1] \to \oO_P^{\rho_\gamma}$.

In \cite[Theorem~4.5]{BaeKoolParkI}, using these data, the authors define reduced fundamental cycles
$$[P]\virt_\mathrm{red} = [P_v^{(q)}(W)]\virt_{\mathrm{red}} \in A_{n-\frac{1}{2}\gamma^2 + \frac{1}{2}\rho_\gamma}(W, \bZ[1/2]).$$
These cycles are essential in developing a surface counting theory on Calabi-Yau fourfolds and used to study the variational Hodge conjecture (cf. \cite[Conjecture~1.12]{BaeKoolParkI}).

Applying Theorem~\ref{thm:red for non-degenerate surj cosection}, we can define these classes integrally after choosing an associated compact, oriented d-manifold model $\pP$ to the oriented, $(-2)$-shifted symplectic scheme $\underline{\pP}$.

\begin{thm}
There exist well-defined integral homological reduced virtual fundamental classes 
$$[P]\virt_\mathrm{red} = [P_v^{(q)}(W)]\virt_{\mathrm{red}} = [\pP]\virt_\mathrm{red} \in H_{2n-\gamma^2 + \rho_\gamma}(W,\bZ).$$
\end{thm}

\begin{rmk}
In fact, by \cite[Theorem~1.16]{BaeKoolParkI}, the class $[P]\virt_\mathrm{red}$ may be alternatively defined directly by using an oriented, $(-2)$-shifted symplectic derived enhancement $\underline{\pP}_\mathrm{red}$ of $P$ which induces the reduced obstruction theory $\mathbb{E}^\bullet_{\mathrm{red}}$ on $P$. It is easy to check that a model for the associated compact, oriented d-manifold is $\pP_\mathrm{red}$ and hence the virtual fundamental class is $[\pP_\mathrm{red}]\virt$, which by~\eqref{eq:loc 5.15} equals $[\pP]\virt_\mathrm{red}$ in the above.
\end{rmk}

\subsection{Gopakumar--Vafa type invariants of holomorphic symplectic fourfolds} Let $W$ be a holomorphic symplectic variety of dimension $4$, i.e. a smooth projective four-dimensional variety with a non-degenerate holomorphic $2$-form $\omega_W \in H^0(W, \Omega_W^2)$. Fix $\beta \in H_2(W,\bZ)$ and $n\in \bZ$.

In \cite{COT2}, the authors consider moduli spaces $P_n^t(W,\beta)$ of pairs $(F,s)$, where $F$ is an one-dimensional coherent sheaf on $W$ with $[F] = \beta$ and $\chi(F)=n$ and $s \in H^0(W,F)$. These pairs are taken to be stable with respect to a stability condition that depends on the choice of real number $t \in \bR$ and ample divisor on $W$. 

For a general choice of $t\in \bR$, the moduli space $P_n^t(W,\beta)$ is a projective scheme and admits a natural $(-2)$-shifted symplectic derived enhancement together with an orientation and a surjective isotropic cosection. Write $\bI = (\oO \to \bF)$ for the universal stable pair and $\pi_W, \pi_P$ for the two projections out of $W \times P_n^t(W,\beta)$.

In order to define counting invariants, consider the insertion operator
\begin{align} \label{eq:insertion operator}
\tau \colon H^m(W,\bZ) \lr H^{m-2}(P_n^t(W,\beta), \bZ) \\
\notag \gamma \mapsto (\pi_P)_\ast\left( \pi_W^\ast\gamma \cup \ch_3(\bF) \right).
\end{align}
For generic $t \in \bR$ and classes $\gamma_i \in H^{m_i}(W,\bZ)$ ($1 \leq i \leq l)$, the $Z_t$-stable pair invariants are defined by the formula
\begin{align*}
P^t_{n,\beta}(\gamma_1, \ldots, \gamma_l) = \int_{[P^t_n(W,\beta)]\virt_{\mathrm{red}}} \prod_{i=1}^l \tau(\gamma_i)
\end{align*}
and, for $n=-1$, where no insertions are necessary,
\begin{align*}
P^t_{-1,\beta} = \int_{[P^t_{-1}(W,\beta)]\virt_{\mathrm{red}}} 1.
\end{align*}

In \cite[Conjecture~1.10]{COT2}, these invariants are conjectured to recover curve-counting invariants on $W$ of Gopakumar--Vafa type. In particular, the conjecture necessitates that they are integral. In the formulation of the stable pair invariants in \cite{COT2}, the authors used the algebraic reduced virtual fundamental cycle defined in \cite{KiemPark}. By Theorem~\ref{thm:isotropic surj cosection}, we may instead equivalently use the integral homological virtual fundamental class $[P^t_n(W,\beta)]\virt_{\mathrm{red}} \in H_\ast(P^t_n(W, \beta), \bZ)$. As a corollary, we obtain the following integrality result, which serves as a sanity check for the conjecture.

\begin{thm}
$P^t_{n,\beta}(\gamma_1, \ldots, \gamma_l), P^t_{-1,\beta} \in \bZ$.
\end{thm}

In a related vein, as in \cite{COT1}, consider the (projective) moduli scheme $M_\beta$ of one dimensional stable sheaves  $F$ on $W$ satisfying $\ch_3(F) = \beta$ and $\chi(F) = 1$. Let $\mathbb{F}_{\mathrm{norm}}$ be the normalized universal sheaf, i.e. $\det (\pi_M)_\ast \mathbb{F}_{\mathrm{norm}} \cong \oO_{M_\beta}$.

It is shown in \cite{COT1} that $M_\beta$ admits a natural, oriented $(-2)$-shifted symplectic derived enhancement which is equipped with an isotropic cosection. 

Using $\mathbb{F}_{\mathrm{norm}}$, there are natural insertion operators $\tau_0$ defined by formula~\eqref{eq:insertion operator}, and for any classes $\gamma_1, \ldots, \gamma_n \in H^\ast(W,\bZ)$ we have associated DT4 invariants  (see equation (0.4) in \cite{COT1})
\begin{align*}
\langle \tau_0(\gamma_1), \ldots, \tau_0(\gamma_n) \rangle_\beta^{\mathrm{DT4}} = \int_{[M_\beta]\virt_{\mathrm{red}}} \prod_{i=1}^n \tau_0(\gamma_i).
\end{align*}

When $\beta$ is an effective curve class, \cite[Conjecture~0.5]{COT1} states that these DT4 invariants recover the genus $0$ Gopakumar--Vafa type invariants of $W$, which are integral. Again, by Theorem~\ref{thm:isotropic surj cosection}, we can use the integral homological virtual fundamental class in the above formula to obtain the following integrality result.

\begin{thm}
$\langle \tau_0(\gamma_1), \ldots, \tau_0(\gamma_n) \rangle_\beta^{\mathrm{DT4}} \in \bZ$.
\end{thm}

\subsection{Integrality of algebraic cosection localized virtual fundamental cycles} We finally comment that using Theorems~\ref{thm:cosection localization for derived schemes any cosection}, ~\ref{thm:main thm}, ~\ref{thm:isotropic surj cosection}, one can define an integral cosection localized or reduced virtual fundamental class in homology, whenever a corresponding algebraic virtual fundamental cycle exists with $\bZ[1/2]$-coefficients. In addition, these classes are well-defined when the virtual dimension is odd in constrast to their algebraic counterparts which vanish in that case.

Several such examples are given in \cite[Section~9]{KiemPark}, among which one was also considered in \cite[Section~7]{SavCosection}.

We briefly sample a few for the convenience of the reader. 
\smallskip

Let $W$ be a Calabi--Yau fourfold with $\omega_W \in H^0(W, K_W)$ a nowhere vanishing form. Let $M$ be a component of the moduli space of simple perfect complexes on $W$ with fixed Chern character $c = (c_0, c_1, c_2, c_3,  c_4)$. This is an algebraic space by \cite{Inaba, Lieblich}, which we assume is a quasi-projective scheme in what follows. As usual, $M$ admits an oriented, $(-2)$-shifted symplectic derived enhancement and an induced symmetric obstruction theory $\mathbb{E}_\bullet \to \bL_M$ with (complex) virtual dimension $\mathrm{vd}$.
\medskip

We first consider \cite[Corollary~9.6]{KiemPark}. Given a $(3,1)$-form $\delta \in H^1(W, \Omega_W^3)$, there is an associated cosection $\sigma^\delta \colon \mathbb{E}_\bullet[1] \to \oO_M$ defined in \cite[Example~9.2]{KiemPark}. If $c_2 = 0$ or a certain bilinear form on $H^1(W, T_W)$ (see equation~(9.7) in \cite[Corollary~9.6]{KiemPark}) vanishes, then $\sigma^\delta$ is isotropic and we get an integral cosection localized virtual fundamental class $[M]_{\mathrm{loc}}\virt \in H_{\mathrm{vd}}(M, \bZ)$. 

On the other hand, if the bilinear form has positive rank $k>0$, then $[M]\virt = 0 \in H_{\mathrm{vd}}(M, \bZ)$ and we can define a reduced virtual fundamental class $[M]_\mathrm{red}\virt \in H_{\mathrm{vd} + 2k}(M, \bZ)$. This generalizes \cite[Corollary~9.6]{KiemPark} in the sense that the algebraic reduced virtual fundamental cycle is defined only for even values of $k$.

We may perform identical steps for cosections associated to $(0,2)$-forms $\gamma \in H^2(W, \oO_W)$ to obtain an integral version of \cite[Corollary~9.7]{KiemPark} and drop the condition that $h^{0,2}$ is a positive and even number in order to define reduced integral classes.
\medskip

For a second example, consider the moduli space $M = P_n(W,\beta)$ of stable pairs $(F,s)$ on $W$ satisfying $[F]=\beta$, $\chi(F)=n$ and assume that $W$ is hyperk\"{a}hler, as in \cite{CaoMaulikToda}. 

Then, by \cite[Theorem~7.2]{SavCosection} and \cite[Example~9.11(i)]{KiemPark}, the homological and algebraic virtual fundamental class of $M$ vanish. Moreover, by \cite[Example~9.11(i)]{KiemPark}, there is an algebraic reduced virtual fundamental cycle $[M]\virt_{\mathrm{red}} \in A_{\frac{1}{2}\mathrm{vd}+1}(M, \bZ[1/2])$, whose image in homology under the cycle class map equals the image of the integral homological virtual fundamental  class $[M]\virt_{\mathrm{red}} \in H_{\mathrm{vd}+2}(M, \bZ)$.


\bibliography{Master}
\bibliographystyle{alpha}

\end{document}